\documentclass{article}

\usepackage{psfrag}
\usepackage{graphicx}
\usepackage{amsmath,amsfonts,amssymb,amsxtra,amsthm,mathrsfs,enumerate,eqnarray}
\usepackage{pinlabel}
\usepackage{color}
\usepackage{tikz-cd}
\usepackage{hyperref}

\usepackage{wrapfig}

\tikzset{
math to/.tip={Glyph[glyph math command=rightarrow]},
loop/.tip={Glyph[glyph math command=looparrowleft, swap]},
loop'/.tip={Glyph[glyph math command=looparrowleft]},
 weird/.tip={Glyph[glyph math command=Rrightarrow, glyph length=1.5ex]},
  pi/.tip={Glyph[glyph math command=pi, glyph length=1.5ex, glyph axis=0pt]},
}

\usepackage{scalerel,stackengine}
\stackMath
\usepackage{verbatimbox} 
\usepackage{xparse}
\newlength\glyphwidth
\newlength\widthofx
\newsavebox\hatglyphCONTENT
\sbox\hatglyphCONTENT{%
    \addvbuffer[-0.05ex -1.3ex]{$\hat{\phantom{.}}$}%
}
\newcommand\hatglyph{\resizebox{0.6\widthofx}{!}{\usebox{\hatglyphCONTENT}}}
\newcommand\shifthat[2]{%
    \stackengine{0.2\widthofx}{%
        \SavedStyle#2}{%
        \rule{#1}{0ex}\hatglyph}{O}{c}{F}{T}{S}%
}
\ExplSyntaxOn
\newcommand\relativeGlyphOffset[1]{%
    \str_case:nnF{#1}{%
        {A}{0.18}%
        {B}{0.1}%
        {W}{0.02}%
        {J}{0.18}%
        {\phi}{0.17}%
    }{0.05}
}\ExplSyntaxOff
\NewDocumentCommand{\hatt}{mO{#1}}{%
    \ThisStyle{%
        \setlength\glyphwidth{\widthof{$\SavedStyle{}\longleftarrow$}}%
        \setlength\widthofx{\widthof{$\SavedStyle{}x$}}%
        \shifthat{\relativeGlyphOffset{#2}\glyphwidth}{#1}%
  }%
}

\DeclareMathOperator{\rank}{rk}

\def\immerses{\looparrowright}
\def\onto{\twoheadrightarrow} \def\into{\hookrightarrow}

\def\ncl#1{\mathord{\langle}\mskip -4mu plus 0mu minus 0mu
  \mathord{\langle}#1\mathord{\rangle}\mskip -4mu plus 0mu minus 0mu
  \mathord{\rangle}}

\def\real#1{\boldsymbol #1}
\def\rk#1{\rank(#1)}

\newcommand{\comm}[1]{\marginpar{\tiny #1}}

\newtheorem{theorem}{Theorem}[section]

\newtheorem{lemma}[theorem]{Lemma}
\newtheorem{corollary}[theorem]{Corollary}

\newtheorem{conjecture}[theorem]{Conjecture}

\theoremstyle{definition}
\newtheorem{definition}[theorem]{Definition}

\theoremstyle{remark}

\author{Larsen Louder and Henry Wilton\footnote{Supported by EPSRC Standard Grant EP/L026481/1.}}

\newcommand{\Addresses}{{
  \bigskip
  \footnotesize

  L.~Louder, \textsc{Department of Mathematics, University College London, Gower Street, London WC1E  6BT, UK}\par\nopagebreak
  \textit{E-mail address:} \texttt{l.louder@ucl.ac.uk}

  \medskip

  H. Wilton, \textsc{DPMMS, Centre for Mathematical Sciences, Wilberforce Road, Cambridge CB3 0WB, UK}\par\nopagebreak
  \textit{E-mail address:} \texttt{h.wilton@maths.cam.ac.uk}

}}

\title{One-relator groups with torsion are coherent}

\begin{document}
\maketitle

\begin{abstract}
  We show that any one-relator group $G=F/\ncl{w}$  with
  torsion is coherent -- i.e., that every finitely generated subgroup of
  $G$ is finitely presented -- answering a 1974 question of Baumslag in this case.
\end{abstract}


\section{Introduction}

\begin{definition}
  A group $G$ is \emph{coherent} if every finitely generated subgroup
  of $G$ is finitely presentable.
\end{definition}

A well known question of Baumslag asks whether every one-relator group $F/\ncl{w}$ is coherent \cite[p.\ 76]{baumslag_problems_1974}.    It is a curious feature of one-relator groups that the case with torsion, in which the relator $w$ is a proper power, is often better behaved than the general case; most famously, one-relator groups with torsion are always hyperbolic \cite{newman-spelling}, and Wise proved that one-relator groups with torsion are residually finite, indeed linear \cite{wise_structure_2012}.   In this note we answer Baumslag's question affirmatively for one-relator groups with torsion.  
\begin{theorem}
  \label{thm: coherence}
  If $G$ is a one-relator group with torsion -- that is, $G\cong F/\ncl{w^n}$, for $n>1$ -- then $G$ is
  coherent.
\end{theorem}

In 2003, Wise circulated \cite{wise_coherence_2003}, including a purported proof of the following conjecture, stated as Theorem 4.13 of that paper.  

\begin{conjecture}\label{conj: NPI implies coherence}
If $X$ is a compact 2-complex with non-positive immersions then $\pi_1X$ is coherent.
\end{conjecture}

The reader is referred to, for instance, \cite{louder-wilton} or \cite{helfer-wise} for the definition of non-positive immersions.  On his webpage, Wise acknowledges that there is a gap (found by Mladen Bestvina) in the proof of \cite[Theorem 4.13]{wise_coherence_2003}.   In \cite{wise_coherence_2005}, Wise used Conjecture \ref{conj: NPI implies coherence} (stated as Theorem 1.5 of that paper) in a proof that Theorem \ref{thm: coherence} follows from the Strengthened Hannah Neumann Conjecture. The latter conjecture has more recently been proved independently by Friedman \cite{friedman_sheaves_2015} and Mineyev \cite{mineyev_submultiplicativity_2012}.

In summary, the results of \cite{wise_coherence_2005} are conditional on Conjecture \ref{conj: NPI implies coherence}, which remains open, and therefore Theorem \ref{thm: coherence} was not hitherto known unconditionally.    After our proof was circulated, we learned from Wise that he has also given an unconditional proof of Theorem \ref{thm: coherence} in  \cite{wise_coherence_2018}, a revised version of \cite{wise_coherence_2003}.

Our proof (and also that of \cite{wise_coherence_2018}) uses Wise's $w$-cycles conjecture (Theorem \ref{wcycles}), which was proved independently by the authors \cite{louder-wilton} and by Helfer--Wise \cite{helfer-wise}.

The outline of the proof is as follows.  We realize $G$ as the
fundamental group of a compact, aspherical orbicomplex $X$. Since one-relator groups  are virtually torsion-free, there is a finite-sheeted covering map
$X_0\immerses X$ so that $G_0=\pi_1X_0$ is torsion free.  We then use
the $w$-cycles conjecture to show that, whenever $Y\immerses X_0$ is
an immersion from a compact two-complex without free faces, the number
of two-cells in $Y$ is bounded by the number of generators of
$\pi_1(Y)$.  In the final step, a folding argument expresses an
arbitrary finitely generated subgroup $H$ of $G_0$ as a direct limit
of fundamental groups of 2-complexes with boundedly many 2-cells, and
we deduce that $H$ is finitely presented.

\subsection*{Acknowledgements}

Thanks to Jim Howie and Hamish Short for pointing out an error in Section 4 of an earlier version.

\section{One-relator orbicomplexes}

\label{section: orbicomplexes}

Let $F$ be a finitely generated free group, and $G=F/\ncl{w^n}$ a one-relator group, where $w\in F$ is not a proper power.  In the usual way, $F$ can be realized as the fundamental group of some finite topological graph $\Gamma$, and $w$ by a continuous map $w:S^1\to\Gamma$. (Since we are only interested in $w$ up to conjugacy, we ignore base points.)  Let $D_n\subseteq\mathbb{C}$ be the closed unit disk equipped with a cone point of order $n$ at the origin. The orbicomplex
\[
X= \Gamma \cup_w D_n
\]
provides a natural model for $G$, in the sense that $G$ is the
(orbifold) fundamental group of $X$.  We call $X$ a \emph{one-relator
  orbicomplex}.  (There is a much more general theory of orbicomplexes
-- see, for instance, \cite{haefliger_complexes_1991} or \cite[Chapter
  III.$\mathcal{C}$]{bridson_metric_1999} -- but the one-relator
orbicomplexes defined here are sufficient for our purposes.) When
$n=1$, $X$ is a \emph{one-relator complex}.

 
A map of 2-complexes is a \emph{morphism} if it sends $n$-cells
homeomorphically to $n$-cells, for $n=0,1,2$.  A morphism of
2-complexes $Y\to Z$ is an \emph{immersion} if it is a local
injection; in this case, we write $Y\immerses Z$.  If $Y$ is a
2-complex and $X$ is the one-relator orbicomplex defined above, a continuous map
$Y\to X$ is a \emph{morphism} if it sends vertices to vertices, edges
homeomorphically to edges, and restricts, on each 2-cell, to the
standard degree-$n$ map $p_n:D_1\to D_n$ given by $p_n(z)=z^n$.  A
morphism $Y\to X$ is an \emph{immersion} if it is locally injective
away from the cone points in the 2-cells (again, we write $Y\immerses X$), and a \emph{covering} if it
is locally a homeomorphism except at the cone points.  The next definition plays a crucial
role in our argument.

\begin{definition}[Degree]\label{defn: Degree}
If $f:Y\immerses X$ is an immersion of two-dimensional (orbi)complexes, then the \emph{degree} of $f$, denoted by $\deg f$, is the minimum number of preimages of a generic point in a 2-cell of $X$. That is: if $X$ is a 2-complex, then $\deg f$ is the minimum number of preimages of any point in the interior of a 2-cell of $X$; and if $X$ is a one-relator orbicomplex with 2-cell $D_n$, then $\deg f$ is the number of preimages of any point in the interior of $D_n$ except $0$.
\end{definition}

Every one-relator group is virtually torsion free  \cite{fischer_one-relator_1972}, and it follows that the orbi-complex $X$ is finitely covered by a genuine 2-complex.  This can be seen using the covering theory for complexes of groups developed in \cite{bridson_metric_1999}, but we give a low-tech proof below.

\begin{theorem}[Unwrapped covers of one-relator complexes]\label{thm: Unwrapped cover}
  Let $G= F/\ncl{w^n}$ be a one-relator group with $w$ not a proper
  power, and $X$ the orbicomplex defined above.  Then there is a
  finite-sheeted covering map
  \[
 X_0\immerses X
 \]
 where $X_0$ is a compact, connected 2-complex.
\end{theorem}
\begin{proof}
Let $X'$ be $\Gamma\cup_{w^n} D$, the (genuine) two-complex that is
the result of gluing a disc $D$ to the graph $\Gamma$ along the map
$w^n$.  Note that there is a natural morphism $X'\to X$ that induces
an isomorphism of fundamental groups.  Let $G_0$ be a torsion-free
subgroup of finite index in $G$, and let $X'_0\to X'$ be the
corresponding covering space. Since $w$ has order $n$ in $G$ \cite[Corollary~4.11]{magnus_combinatorial_2004}, the
2-cells of $X'_0$ come in families of cardinality $n$, such that all
the 2-cells in each family have the same attaching map. Let $X_0$ be
the quotient of $X'_0$ obtained by collapsing each family to a single
2-cell. Picking one 2-cell from each family specifies an inclusion
$X_0\into X'_0$, and the quotient map $X'_0\to X_0$ is then visibly a
retraction, and an isomorphism on fundamental groups by the
Seifert--van Kampen theorem. The composition
\[
X_0\into X'_0\immerses X'\to X
\]
is the required covering map.
\end{proof}

We emphasize that the complex $X_0$ in the above theorem is a 2-complex, not just an orbicomplex. That is, the covering map $X_0\to X$ restricts to $p_n$ on each 2-cell.  We will call such a cover $X_0$ \emph{unwrapped}.

\section{A bound on the number of 2-cells}

A two-complex $Y$ is \emph{reducible} if it has a free face.\footnote{Note that this definition is slightly stronger than the definition given in \cite{louder-wilton}, where a 2-complex was called `reducible' if it has a free face or a free edge.  The definition given here is more convenient in this context, since the complexes $Y_i$ produced by Lemma \ref{lem: direct limit} are irreducible in this sense, but not in the sense of \cite{louder-wilton}.} Writing
\[
\partial_Y:\coprod S^1 \to Y^{(1)}
\]
for the disjoint union of the attaching maps of the 2-cells, this means that  there is an edge
$e$ of the 1-skeleton $Y^{(1)}$ such that $\partial_Y^{-1}(e)$ consists of a single edge in $\coprod S^1$.   If $e$ is such an edge and $Y'$ is the 2-complex obtained by collapsing the face of $Y$ incident at $e$, then the natural inclusion map $Y'\into Y$ is a homotopy equivalence, and induces an isomorphism on fundamental groups. Of course, if $Y$ is not reducible it is called \emph{irreducible}.

The main theorem of~\cite{louder-wilton} (or \cite{helfer-wise}) can be restated as a result about immersions to one-relator orbicomplexes, as follows.

\begin{theorem}
  \label{wcycles}
  Let $X$ be a one-relator orbicomplex, $Y$ a finite 2-complex and $f:Y\immerses X$ an immersion.  If $Y$ is irreducible then $\chi(Y^{(1)})+\deg f \leq 0$.
 \end{theorem}
 \begin{proof}
 This follows from \cite[Theorem 1.2]{louder-wilton}, with $\Gamma=X^{(1)}$, $\Gamma'=Y^{(1)}$, $\rho$ the restriction of $f$ to $\Gamma'$, $\Lambda=w$, and $\mathbb{S}$ the disjoint union of the boundaries of the 2-cells of $Y$.  If some edge of $Y^{(1)}$ is not hit by a 2-cell, then we may remove that edge, increasing $\chi(Y^{(1)})$.  Otherwise, $Y$ is reducible in the sense of \cite{louder-wilton}, so \cite[Theorem 1.2]{louder-wilton} applies, taking $\Lambda'=\partial_Y$ and $\sigma$ the induced map from the boundaries of the 2-cells of $Y$ to the boundary of the 2-cell of $X$. 
 \end{proof}
 
Here, we apply Theorem \ref{wcycles} to relate the number of 2-cells of $Y$ to the rank of its fundamental group. (By the \emph{rank} of a group, we mean the minimal cardinality of a generated set.)

\begin{corollary}
  \label{cor: two cells bound}
 Let $f\colon Y\immerses X$ be an immersion from a finite, irreducible 2-complex $Y$ to a one-relator orbicomplex $X$. Then
  \[
    \chi(Y)+(n-1)\vert\{\mbox{\emph{2-cells in }}Y\}\vert\leq 0~.
  \]
  In particular,
  \[
    \vert\{\mbox{\emph{2-cells in }}Y\}\vert\leq
    \frac{\rk{\pi_1(Y)}-1}{n-1}<\rk{\pi_1(Y)}
  \]
  if $n>1$.
\end{corollary}
\begin{proof}
  By Theorem~\ref{wcycles}, $\chi(Y^{(1)})+\deg(f)\leq 0$. Since $f$ restricts to $p_n$ on each 2-cell of $Y$, it follows that
  \[
    \deg f=n\vert\{\mbox{2-cells in }Y\}\vert
  \]
  since $X$ is one-relator.

  The Euler characteristic of $Y$ is the Euler characteristic of $Y^{(1)}$ plus the
  number of two-cells in $Y$, so Theorem~\ref{wcycles} implies the
  first assertion.  The second assertion now follows from the first, using the trivial fact that
   \[
  1-\rk{\pi_1Y}\leq~\chi(Y)
  \]
  for any connected 2-complex $Y$.
\end{proof}

In the case with torsion, Corollary \ref{cor: two cells bound} gives a bound on the the number of 2-cells of an immersion in terms of the rank of the fundamental group. In order to make a connection to arbitrary finitely generated subgroups of $G$, we use folding, in the spirit of Stallings.

\section{Folding}

\label{section: folding}

Folding was introduced by Stallings to study free groups and their subgroups. The next lemma extends \cite[Algorithm 5.4]{stallings-folding} to the context of 2-complexes and their morphisms.

\begin{lemma}\label{lem: Folding 2-complexes}
  Any morphism of finite 2-complexes $A\to B$ factors as 
  \[
    A\to C\immerses B
  \]
  where $A\to C$ is surjective and $\pi_1$-surjective.  Furthermore, if $A\to B$ factors as
  \[
  A\to D\immerses B
  \]
  then there is a unique immersion $C\immerses D$ so that the following diagram commutes.
  \begin{center}
  \begin{tikzcd}
 A\arrow{rr}\arrow{dr} && D\arrow[loop-math to]{r} & B \\
  & C\arrow[loop-math to,dotted]{ur}{\exists!}\arrow[loop-math to, bend right=15]{urr} & 
  \end{tikzcd}
\end{center}
  
      \end{lemma}
\begin{proof}
  Folding shows that the map of 1-skeleta factors as
  \[
    A^{(1)}\to C^{(1)}\immerses B^{(1)}
  \]
  where $A^{(1)}\to C^{(1)}$ is surjective and $\pi_1$-surjective. We
  now construct $C$ by pushing the attaching maps of the 2-cells of
  $A$ forward to $C^{(1)}$ and identifying any 2-cells with the same
  image in $B$ and equal attaching maps. The resulting map $A\to C$ is
  surjective and $\pi_1$-surjective. We next check that the
  natural map $C\to B$ is an immersion.

  Since $C\to B$ is a morphism, it can only fail to be locally
  injective at a point $p\in C$ if two higher-dimensional cells
  incident at $p$ have the same image in $B$.  The map of
  1-skeleta is an immersion, so this can only occur if two 2-cells
  $c_1,c_2$ in $C$, incident at $p$, have the same image in $B$.
  Because the attaching maps of $c_1$ and $c_2$ agree at $p$ and
  $C^{(1)}\to B^{(1)}$ is an immersion, it follows that the attaching
  maps of $c_1$ and $c_2$ agree everywhere. Therefore, $c_1$ and $c_2$
  are equal in $C$ by construction.
  
It remains to prove the universal property.  This fact is standard for graphs, which defines the required immersion of 1-skeleta $C^{(1)}\immerses D^{(1)}$ uniquely. Let $c$ be a 2-cell in $C$ and let $a_1,a_2$ be preimages of $c$ in $A$.  By construction, $a_1$ and $a_2$ have the same boundary in $C^{(1)}$ and the same image in $B$.  Therefore, their images $d_1, d_2$ respectively in $D$ have the same boundary in $D^{(1)}$ and the same image in $B$. But $D\immerses B$ is an immersion, and it follows that $d_1=d_2$. Therefore, we may extend the map $C^{(1)}\to D^{(1)}$ across $c$ in a unique way, as required.
\end{proof}

A \emph{free edge} of a 2-complex is an edge of the 1-skeleton that is not in the image of the attaching map of any 2-cell.   The next result appeals to the Scott lemma (which plays a crucial role in the proof of coherence for 3-manifold groups) to represent a finitely generated, freely indecomposable subgroup by an immersion from a compact 2-complex without free edges or faces.

\begin{lemma}\label{lem: Representing subgroups by immersions without free faces or edges}
Let $X$ be a 2-complex, $G=\pi_1(X)$, and $H$ a non-trivial, finitely generated, freely indecomposable subgroup of $G$.   There is an immersion from a compact irreducible 2-complex without free edges $f:Y\immerses X$ such that $f_*\pi_1Y$ is conjugate to $H$.
\end{lemma}
\begin{proof}
By the Scott lemma  \cite[Lemma 2.2]{scott_finitely_1973}, there is a surjection from a finitely presented group $H'\onto H$ that does not factor through a free product.  Since $H'$ is finitely presented, there is a morphism of combinatorial 2-complexes $Z'\to X$ that represents the composition $H'\to H \to G$.  By Lemma \ref{lem: Folding 2-complexes}, this morphism factors as
\[
Z'\to Z\immerses X
\]
where $Z'\to Z$ is $\pi_1$-surjective.  By the conclusion of the Scott lemma, $\pi_1Z$ is freely indecomposable.  Let $Y'\subseteq Z$ be a deformation retract of $Z$ obtained by iteratively collapsing free faces.  Since $\pi_1Y'=\pi_1Z$ is freely indecomposable, any free edges of $Y'$ are separating, and one complementary component of each free edge is simply connected. Let $Y$ be the unique non-simply-connected component obtained from $Y'$ by deleting free edges, and let $f:Y\immerses X$ be the natural immersion.  Then $Y$ is as required: $f_*\pi_1Y$ is conjugate to $f_*\pi_1Y'=H$, and $Y$ has neither free faces nor free edges.
\end{proof}

The next two lemmas show that a finitely generated subgroup can be represented by a direct limit of immersions of irreducible 2-complexes. We start with the freely indecomposable case.

\begin{lemma}
  \label{lem: Freely indecomposable direct limit}
  Let $X$ be a 2-complex, $G=\pi_1(X)$, and $H\leq G$ a finitely
  generated, freely indecomposable subgroup. Then there is a sequence of $\pi_1$-surjective immersions of
  compact, connected two-complexes 
\[
Y_0\immerses Y_1\immerses\cdots\immerses Y_i\immerses\cdots X
\]
 with the following properties.\begin{enumerate}[(i)]
  \item Each $Y_i$ is irreducible.
  \item  $H=\varinjlim\pi_1 Y_i$
 \item The number of free edges of $Y_i$ is uniformly bounded.
    \end{enumerate}
\end{lemma}
\begin{proof}
If $H$ is trivial then so is the result.  Otherwise, let $f:Y\immerses X$ be the immersion guaranteed by Lemma \ref{lem: Representing subgroups by immersions without free faces or edges}, so that $f_*\pi_1Y$ is conjugate to $H$.  Let $f_0:Y_0\immerses X$ be a wedge of $Y$ with an interval $I$ so that, fixing a basepoint at the end of the interval, $f_{0*}\pi_1Y_0=H$. Let $(r_i)$ be an enumeration of representatives of the conjugacy  classes of $\ker f_{0*}$, where each $r_i$ is an immersed combinatorial loop in the 1-skeleton of $Y\subseteq Y_0$.  

We now construct the immersions $Y_i\immerses Y_{i+1}$ inductively, assuming that $\{r_0,\ldots, r_i\}$ represent elements of the kernel of $\pi_1Y_i\to G$.   Let $E_{i+1}\to X$ be a reduced van Kampen diagram for $r_{i+1}$, and let $Z=Y_i\cup_{r_{i+1}}E_{i+1}$.  Since $r_{i+1}$ did not cross any free edges of $Y_i$, $Z$ does not have any free faces.  We now apply Lemma \ref{lem: Folding 2-complexes} to the natural map $Z\to X$, which yields
\[
Z\to Y_{i+1}\immerses  X
\]
for the desired 2-complex $Y_{i+1}$.    Next, we prove properties (i), (ii) and (iii).

By construction, $Y_0$ has no free faces. Therefore, we may prove (i) by induction: assuming that $Y_i$ has no free faces, we claim that $Y_{i+1}$ has no free faces.  Suppose by way of contradiction that an edge $e$ is a free face of a 2-cell $c$ in $Y_{i+1}$. Since $Y_i$ immerses into $Y_{i+1}$, and $Y_i$ has no free faces, it follows that $e$ is not in the image of $Y_i$. Therefore, $e$ is the image of an edge $e'$ in the interior of $E_{i+1}$. The two neighbouring 2-cells of $E_{i+1}$ both map to $c$ folding across $e'$, contradicting the hypothesis that $E_{i+1}$ is reduced.

Property (ii) is immediate by construction.  For (iii), simply note that any free edges of $Y_i$ are the image of an edge in the interval $I$.
\end{proof}

Finally, we deal with possibly freely decomposable subgroups, by appealing to the Grushko decomposition.

\begin{lemma}
  \label{lem: direct limit}
  Let $X$ be a 2-complex, $G=\pi_1(X)$, and $H\leq G$ a finitely
  generated subgroup. Then there is a sequence of $\pi_1$-surjective immersions of
  compact, connected two-complexes 
\[
Y_0\immerses Y_1\immerses\cdots\immerses Y_i\immerses\cdots X
\]
 with the following properties.\begin{enumerate}[(i)]
  \item Each $Y_i$ is irreducible.
  \item  $H=\varinjlim\pi_1 Y_i$
 \item The number of edges of $Y_i$ that are not incident at a 2-cell is uniformly bounded.
    \end{enumerate}
\end{lemma}
\begin{proof}
Let $H=H_0*\cdots*H_p*F_q$ be the Grushko decomposition of $H$.   For each $1\leq r\leq p$, let
\[
Y^r_0\immerses Y^r_1\immerses\cdots\immerses Y^r_i\immerses\cdots X
\]
 be the sequence provided by applying Lemma \ref{lem: Freely indecomposable direct limit} to $H_r$. Let $Z\immerses X$ be a graph immersed in the 1-skeleton with $\pi_1Z=F_q$.  For each $i$, let
 \[
Y'_i=Y^1_i\vee\cdots\vee Y^p_i\vee Z 
 \]
 and let $Y_i\immerses X$ be the immersion obtained by applying Lemma \ref{lem: Folding 2-complexes} to $Y'_i$.  
 
 The required immersion $Y_i\immerses Y_{i+1}$ exists by the universal property of Lemma \ref{lem: Folding 2-complexes}: see the following commutative diagram.
 
 
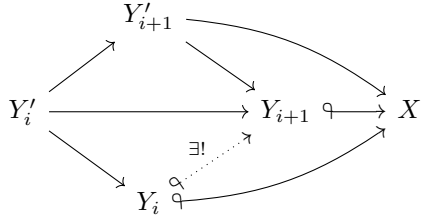
\begin{center}
  \begin{tikzcd}
 & Y'_{i+1} \arrow{dr}\arrow[bend left=15]{drr}& \\ 
 Y'_i\arrow{rr}\arrow{dr}\arrow{ur} && Y_{i+1}\arrow[loop-math to]{r} & X \\
  & Y_{i}\arrow[loop-math to,dotted]{ur}{\exists!}\arrow[loop-math to, bend right=15]{urr} & 
  \end{tikzcd}
 \end{center}
 Properties (i), (ii) and (iii) are clear from the construction.  
\end{proof}

We are now ready to prove our main result.

\begin{proof}[Proof of Theorem \ref{thm: coherence}]
Realize $G$ as the fundamental group of a one-relator orbicomplex $X$.  Let $H$ be a finitely generated subgroup of $G$. Let $G_0\leq G$ be  a torsion-free subgroup of finite index, corresponding to the unwrapped cover $X_0\immerses X$ provided by Theorem \ref{thm: Unwrapped cover}.  Since a finite extension of a finitely presented group is finitely presented, we may replace $H$ by $H\cap G_0$, and so assume that $H\leq G_0$.  Consider the sequence of immersions
\[
Y_0\immerses Y_1\immerses\cdots\immerses Y_i\immerses\cdots X_0
\]
provided by Lemma \ref{lem: direct limit}, taking $X_0$ for $X$.  By Corollary \ref{cor: two cells bound}, the number of 2-cells of each $Y_i$ is bounded. Each 2-cell of $Y_i$ is a copy of the unique 2-cell of $X$, hence has boundary of bounded length.  Combining this with item (iii) of Lemma \ref{lem: direct limit}, we see that the number of 1-cells (and hence also 0-cells) of $Y_i$ is also bounded. Since $X_0$ is finite, there are only finitely many combinatorial types of immersions $Y_i\immerses X_0$. Because $Y_i\immerses X_0$ factors through $Y_{i+1}\immerses X_0$, there is an infinite subsequence
\[
Y_{i_1}\immerses Y_{i_2}\immerses\cdots \immerses Y_{i_j}\immerses\cdots X_0
\]
so that each map $Y_{i_j}\immerses Y_{i_{j+1}}$ is a homeomorphism;
therefore, $H=\varinjlim \pi_1Y_{i_j}=\pi_1(Y_{i_1}) $ is finitely
presented, as required.
\end{proof}

\section{Groups with good stackings}

The results of \cite{louder-wilton} apply equally well to a class of
groups which is rather larger than the class of one-relator
groups.

\begin{definition}[Stacking]\label{defn: Good stacking}
Let $X$ be a 2-dimensional orbicomplex and let
\[
\Lambda: \coprod S^1\to X^{(1)}
\]
be the coproduct of the attaching maps of the 2-cells.  A \emph{stacking} of $X$ is a lift of $\Lambda$ to an embedding
\[
\widehat{\Lambda}:\mathbb{S}\equiv\coprod S^1\into X^{(1)}\times\mathbb{R}~; 
\]
write $\widehat{\Lambda}(x)=(\Lambda(x),h(x))$.  A stacking is called \emph{good} if, for each component $S$ of the domain of $\Lambda$, there is a point $a\in S$ so that $h(a)\geq h(x)$ for all $x\in\mathbb{S}$ with $\Lambda(a)=\Lambda(x)$, and there is also a point $b\in S$ so that so that $h(b)\leq h(x)$ for all $x\in\mathbb{S}$ with $\Lambda(b)=\Lambda(x)$.
\end{definition}

The results of \cite{louder-wilton} apply to the fundamental groups of orbicomplexes with good stackings. 

\begin{definition}\label{defn: Branched good stacking}
We say that a group has a \emph{good stacking} if it is the fundamental group of a compact, 2-dimensional orbicomplex that admits a good stacking.  We say it has a \emph{good branched stacking} if it has a good stacking, and every 2-cell has a cone point of index at least 2.
\end{definition}

Every one-relator group admits a good stacking \cite[Lemma 3.4]{louder-wilton}, which is branched if the group has torsion.  Corollary \ref{cor: two cells bound} applies to groups with a good branched stacking.  The proof of Theorem \ref{thm: coherence} applies verbatim to groups with good branched stackings, except that groups with good branched stackings are not known to admit unwrapped covers -- that is, the analogue of Theorem \ref{thm: Unwrapped cover} is unknown.

However, we conjecture that Wise's proof that one-relator groups with torsion are residually finite goes through for groups with branched good stackings.

\begin{conjecture}\label{conj: Qc hierarchy}
If $G$ has a good branched stacking then $G$ is hyperbolic, and has a virtual quasiconvex hierarchy, in the sense of \cite{wise_structure_2012}.
\end{conjecture}

By the results of \cite{wise_structure_2012}, Conjecture \ref{conj: Qc hierarchy} would imply that every such group is virtually torsion-free, and hence coherent by the same proof as Theorem \ref{thm: coherence}.

\bibliographystyle{amsalpha}

\providecommand{\bysame}{\leavevmode\hbox to3em{\hrulefill}\thinspace}
\providecommand{\MR}{\relax\ifhmode\unskip\space\fi MR }
\providecommand{\MRhref}[2]{%
  \href{http://www.ams.org/mathscinet-getitem?mr=#1}{#2}
}
\providecommand{\href}[2]{#2}

\Addresses

\end{document}